\documentclass[12pt,a4paper]{amsart}

\usepackage{amssymb,amscd,amsmath,amsthm,url}
\usepackage[all]{xy}
\usepackage[pdftex]{graphicx}
\usepackage{caption}
\usepackage{dsfont}
\usepackage{braket}

\newtheorem {Theorem}                    {Theorem}
\newtheorem {Proposition}[Theorem]       {Proposition}
\newtheorem {Lemma}      [Theorem]       {Lemma}
\newtheorem {Corollary}  [Theorem]       {Corollary}
\newtheorem*{Claim}       {Claim}
\newtheorem*{Question}   {Question}

\theoremstyle{definition}
\newtheorem  {Definition} [Theorem]{Definition}

\theoremstyle{remark}
\newtheorem*{Remark}	{Remark}
\newtheorem {Remark2} [Theorem] {Remark}

\numberwithin{Theorem}{section}
\newcommand{\mute}[1] {}

\newcommand{\mc}[1]{\mathcal{#1}}

\def \R{\mathbb{R}}

\def \del{\partial}

\def \ssm{\smallsetminus}
\usepackage{romannum}
\usepackage{enumitem}
\usepackage{microtype}

\graphicspath{{./Images/}}
\def\eps{\varepsilon}
\begin{document}
\pagenumbering{arabic}
\author[A. Bar-Natan]{Assaf Bar-Natan \\McGill University, 805 Sherbrooke St. W, Montreal, QC, H3A 0B9}
\title{Arcs on Punctured Disks Intersecting at Most Twice with Endpoints on the Boundary}
\maketitle

\section*{Abstract}
Let $D_n$ be the $n$-punctured disk. We prove that a family of 
essential simple arcs starting and 
ending at the boundary and pairwise intersecting 
at most twice is of size at most $\binom{n+1}{3}$. On the way, 
we also show that any nontrivial square complex 
homeomorphic to a disk whose 
hyperplanes are simple arcs intersecting at most twice must 
have a corner or a spur.

\tableofcontents
\section{Introduction}
It is a classical theorem that maximal cliques in the curve graph 
$\mc{C}^0(S)$ of a hyperbolic closed orientable surface $S$ 
are of size $\frac{3}{2}|\chi|$ (see \cite{Farb}, for example).

We define the augmented curve graph $\mc{C}^k(S)$ as follows:
\begin{itemize}
  \item The vertices of $\mc{C}^k(S)$ are homotopy classes of 
    essential simple closed curves on $S$. 
  \item Two vertices are connected by an edge if they have 
    representatives intersecting at most $k$ times.
\end{itemize}
The sizes of cliques of these graphs were explored 
in \cite{Juvan}, \cite{Malestein}, \cite{Przytycki} and, 
more recently, in \cite{Aougab}. The combination of these papers show 
that $N_1(S)$, the maximal clique size in $\mc{C}^1(S)$ satisfies:
\begin{align*}
O(|\chi|^2)\le N_1(S) \le O\left(\frac{|\chi|^3}{(\log|\chi|)^2}\right)
\end{align*}

In the case of arcs, the augmented arc graph, $\mc{A}^k(S)$ 
is similarly defined for punctured surfaces. 
P.~Przytycki \cite{Przytycki} showed that 
this arc graph has maximal clique size growing asymptotically as 
$\Theta(|\chi|^{k+1})$. In particular, for $k=1$ it has 
maximal clique size $2|\chi|(|\chi|+1)$. When $S$ is a punctured sphere, 
the maximal clique size of $\mc{A}^2(S)$ is known to be 
$|\chi|(|\chi|+1)(|\chi|+2)$, see \cite{Smith}. When the 
endpoints of arcs are fixed, then we get a subgraph
of $\mc{A}^1(S)$, whose maximal clique size 
is calculated in \cite{Przytycki} 
to be $\binom{|\chi|+1}{2}$. The main result of this paper will 
be to compute the maximal clique size of the corresponding 
subgraph of $\mc{A}^2(S)$. 
\begin{Definition}
  Let $D_n$ be the $n$-punctured disk. 
  An \emph{arc} in $D_n$ is an embedding $\alpha:(0,1)\to D_n$ 
  converging to punctures or the boundary $\del D_n = S^1$ at 
  $0,1$. We often identify arcs with their images in $D_n$. 
  We write $\alpha\sim \beta$ if $\alpha$ is homotopic 
  to $\beta$ relative to $\del D_n$. In this article, 
  we assume that all arcs are \emph{essential}, meaning that 
  they are not homotopic to a constant arc. 
\end{Definition}
\begin{Definition} Two arcs $\alpha,\beta$ are said to be 
  in \emph{minimal position} if they minimize $|\alpha\cap\beta|$ 
  in their respective homotopy classes relative to $\del D_n$.
\end{Definition}
\begin{Definition}
  A family of simple arcs $\mc{A}$ on $D_n$ is called \emph{good} if:
  \begin{enumerate}
    \item For any arc $\alpha\in\mc{A}$, the endpoints of $\alpha$ 
      lie in $\del D_n$.
    \item For any arcs $\alpha,\beta\in\mc{A}$, $\alpha\not\sim \beta$.
    \item For any $\alpha,\beta\in\mc{A}$, $|\alpha\cap\beta|\le 2$.
    \item The arcs in $\mc{A}$ are in minimal position.
  \end{enumerate}
\end{Definition}
We are now ready to state the main result of this paper.
\begin{Theorem} \label{Main_Theorem}
  The maximal cardinality of a good family of arcs on $D_n$ 
  is $\binom{n+1}{3}$.
\end{Theorem}
In order to prove this theorem, we will first obtain a 
result on square complexes, which is of independent interest:
\begin{Definition}
  Let $X$ be a finite planar square complex homotopy equivalent to a disk. 
  \begin{enumerate}
    \item A \emph{corner} at $v$ of $X$ consists of two consecutive 
      boundary edges intersecting at $v$ which belong to a single square.
    \item A \emph{spur} (at $v$) of $X$ is a $1$-cell 
      with a vertex $v$ which is not 
      contained in any other cell of $X$.
  \end{enumerate}
\end{Definition}
\begin{Theorem} \label{Square_Cplx_Thm}
  Let $X$ be a finite planar square complex homotopy 
  equivalent to a disk. Suppose that 
  the hyperplanes in $X$ are simple arcs, pairwise intersecting 
  at most twice, and that $X$ is not a single $0$-cell. 
  Then $X$ has a corner or a spur.
\end{Theorem}
If hyperplanes in $X$ are simple arcs pairwise intersecting at most twice, 
then we say that they \emph{satisfy condition ($\ast$)}.

\section{Proof Plan}
In this section, we provide an outline for the proof of 
Theorem~\ref{Main_Theorem}. 
\begin{Definition}
  Let $\mc{A}$ be a good family of arcs. An arc 
  $\alpha\in \mc{A}$ is called \emph{isolated} if $\alpha$ 
  is disjoint from any other arc $\beta\in \mc{A}$, and splits 
  $D_n$ into two components, one of which contains a single puncture. 
\end{Definition}
The following two 
propositions will be used to prove Theorem~\ref{Main_Theorem}.
\begin{Proposition} \label{Cloud_Prop}
  In any maximal good family of arcs 
  $\mathcal{A}$ there exists an isolated arc ~$\alpha$.
\end{Proposition}
\begin{Proposition} \label{Bigon_Count} Assume that 
  $\alpha_p\in \mc{A}$ is an isolated arc, and let $p$ be the 
  unique puncture in one of the connected components of $D_n\ssm \alpha_p$. 
  Let $\mc{A}'$ be the set of homotopy classes of essential arcs 
  obtained from $\mc{A}$ by removing $p$. Then
  \begin{align*}
    |\mc{A}| - |\mc{A}'| \le \binom{n}{2}
  \end{align*}
\end{Proposition}
Succinctly, the deduction arguments of this paper are as follows:
\begin{align*}
  Theorem~\ref{Square_Cplx_Thm} 
  \Rightarrow &Proposition~\ref{Cloud_Prop}\\
  Proposition~\ref{Bigon_Count}+ &Proposition~\ref{Cloud_Prop} 
  \Rightarrow Theorem~\ref{Main_Theorem}
\end{align*}

\begin{proof} [Proof of Theorem~\ref{Main_Theorem} from the propositions]
  The proof follows by induction. The base case, where 
  $n=1$ is obvious. By Proposition~\ref{Cloud_Prop}, there 
  exists an isolated arc $\alpha_p$. 
  After removing $p$ and identifying arcs that become 
  homotopic to each other, we will 
  get a new family of arcs $\mathcal{A}'$ on $D_{n-1}$. By induction, 
  $|\mathcal{A}'| \le \binom{n}{3}$, and thus, 
  by Proposition~\ref{Bigon_Count},
  \begin{align*}
    |\mathcal{A}| \le |\mathcal{A}'| + \binom{n}{2}\le \binom{n+1}{3},
  \end{align*}
  which is what we want to show.
\end{proof}
We now show how Proposition~\ref{Cloud_Prop} follows from 
Theorem~\ref{Square_Cplx_Thm}:
\begin{Definition}
  Let $\mc{A}$ be a good family of arcs. After a small 
  perturbation, we can assume that there are no triple 
  intersections of arcs in $\mc{A}$. Consider the square 
  complex, $X(\mc{A})$, \emph{dual} to $\mc{A}$: each vertex 
  in $X(\mc{A})$ corresponds to a connected component of 
  $D_n\ssm \cup \mc{A}$, and edges correspond to arc segments 
  in the boundary of adjacent components. 
\end{Definition}
\begin{Remark}
  If $\mc{A}$ is any family of arcs starting and ending 
  at $\del D_n$, then $X(\mc{A})$ is homotopy equivalent to 
  a disk, using the retraction as in 
  Figure~\ref{fig:Square_Complex_Retract}
    \begin{center}
      \begin{figure}
        \includegraphics[width=0.15\linewidth]{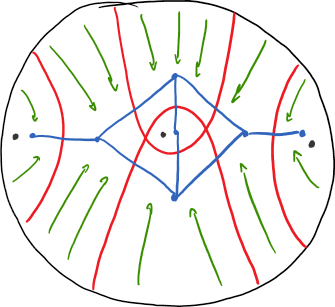}
        \caption{A retraction from the disk to $X(\mc{A})$ 
        where $\mc{A}$ is a good family (in red) on $D_3$ (punctures 
        in black).}
        \label{fig:Square_Complex_Retract}
      \end{figure}
    \end{center}
\end{Remark}
\begin{Lemma} \label{Corner_Max}
  \begin{enumerate}
    \item If $X(\mc{A})$ has a spur, then $\mc{A}$ has an isolated arc 
      or is not maximal.
    \item If $X(\mc{A})$ has a corner, then $\mc{A}$ is not maximal.
  \end{enumerate}
\end{Lemma}
\begin{proof}
  \begin{enumerate}
    \item A spur corresponds to a region $A$ that is adjacent to a 
      single other region $B$ along a single arc $\alpha$. This 
      means that $\del A \cap \del B=\alpha$, and thus, 
      $\alpha$ is disjoint from any other arc in $\mc{A}$. 
      Since $\alpha$ is essential, 
      $A$ contains at least one puncture. Since $X(\mc{A})$ has a 
      spur terminating at $A$, it follows that $\mc{A}\ssm \alpha$ 
      is contained in $B$. If $A$ has at least two punctures, this 
      contradicts maximality of $\mc{A}$ (see 
      Figure~\ref{fig:Spur_Side_Punctures}).
  \begin{center}
    \begin{figure}
      \includegraphics[width=0.2\linewidth]{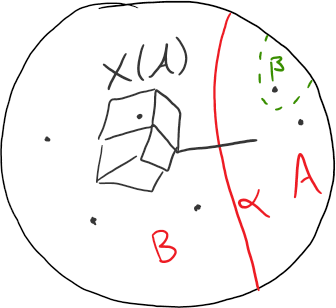}
      \caption{The region $A$ cannot contain more than one puncture, or 
      else the arc $\beta$ in the figure could be added to $\mc{A}$.}
      \label{fig:Spur_Side_Punctures}
    \end{figure}
  \end{center}
    \item Let $e_1,e_2$ be the edges forming a corner at a vertex $v$ in 
      $X(\mc{A})$. Let $h_1,h_2$ be the hyperplanes determined by 
      $e_1,e_2$ (see Definition~\ref{Hyperplane_Def}). 
      Consider the region in $D_n\ssm \cup \mc{A}$ corresponding to $v$, 
      call it $A_v$ (see Figure~\ref{fig:Corner_Contradicts_Min}).
  \begin{center}
    \begin{figure}
      \includegraphics[width=0.3\linewidth]{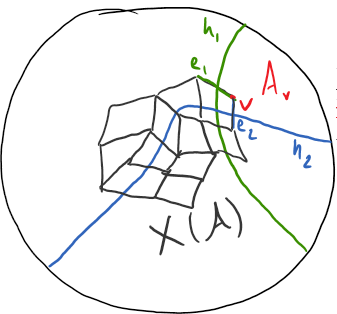}
      \caption{The region $A_v$ defined from the corner at $v$.}
      \label{fig:Corner_Contradicts_Min}
    \end{figure}
  \end{center}
      If $A_v$ does not contain a puncture, then the arcs 
      corresponding to $h_1,h_2$ will not be in minimal position. If 
      $A_v$ does contain a puncture, then the arc $\del A_v\ssm (\del D_n)$ 
      is disjoint from all other arcs in $\mc{A}$ after a homotopy, 
      and can be added to the family. Thus, $\mc{A}$ is not 
      maximal, contradiction.
  \end{enumerate}
\end{proof}
\begin{proof} [Proof of Proposition~\ref{Cloud_Prop} from 
  Theorem~\ref{Square_Cplx_Thm}]
  If $\mc{A}$ has a spur, then we are done by Lemma~\ref{Corner_Max}(1). 
  Otherwise, $X(\mc{A})$ has a corner, which by Lemma~\ref{Corner_Max}(2)
  contradicts maximality.
\end{proof}

\section{Helpful Lemmas}
In this section, we present some helpful lemmas and definitions 
used throughout this paper to prove Proposition~\ref{Bigon_Count}.
\begin{Definition} Here, we regard the punctures as marked 
  points. Let $\alpha,\beta$ be arcs in $D_n$.

  A \emph{bigon} (resp.\ \emph{half-bigon}, \emph{strip}) 
  between $\alpha$ and $\beta$ is 
  a connected component $D$ of $D_n\ssm(\alpha\cup\beta)$ which 
  is 
  \begin{enumerate}
    \item homeomorphic to a disk, 
      \mute{item is not (resp.\ is) adjacent to $\del D_n$, 
      and}
    \item $\alpha\cap \bar D,\beta\cap \bar D$ 
      are nonempty and connected, and $D\cap \del D_n$ is 
      empty (resp.\ nonempty and connected, disconnected).
  \end{enumerate}
    \begin{center}
      \begin{figure}
        \includegraphics[width=0.5\linewidth]{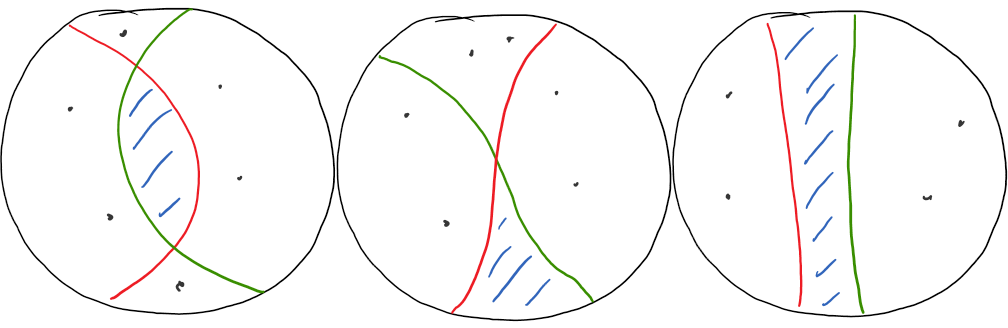}
        \caption{From left to right: a bigon, half bigon, and strip.}
        \label{fig:Bigon_Example}
      \end{figure}
    \end{center}
\end{Definition}
\begin{Definition}
  We say that two arcs $\alpha,\beta$ on $D_n$ are in \emph{minimal 
  position} if they minimize $|\alpha\cap\beta|$ in their 
  respective homotopy classes.
\end{Definition}
\mute{In many cases, it will be convenient to say that a bigon (resp.\
half-bigon, strip) contains a 
puncture $p$. By this we mean that if we remove all punctures from 
$D_n$, then there would be a bigon (resp.\ half-bigon, strip) 
which, when putting the punctures back, would contain $p$.}
\begin{Lemma} \label{Bigon_Criterion}
  (The Bigon Criterion, Lemma 1.7 in \cite{Farb})
  Two arcs are in minimal position if and only if they intersect 
  transversely, and there are 
  no bigons or half-bigons between them.
\end{Lemma}
\begin{Corollary}\label{One_Intersection_Lemma}
  Let $\alpha$ be an arc with endpoints on $\del D_n$ and $\beta$ 
  an arc with endpoints outside of $\del D_n$. If $\alpha$ 
  and $\beta$ intersect once and transversely, then 
  they are in minimal position.
\end{Corollary}
\begin{Definition}
  We say that an arc $\alpha$ \emph{separates} two punctures $q,r$ if 
  $D_n\ssm \alpha$ is disconnected and $p$ and $r$ lie on 
  different connected components of $D_n\ssm \alpha$.
\end{Definition}
\begin{Remark2} \label{Essential_Separates}
  An arc with endpoints on $\del D_n$ is essential if and 
  only if it separates some punctures.
\end{Remark2}
\mute{
\begin{Definition}
  An \emph{isolated puncture} $p$ with respect to a family 
  of arcs $\mc{A}$ is a puncture in $D_n$ such 
  that there exists an arc in $\mc{A}$, call it $\alpha_p$, such that 
  $\alpha_p$ is disjoint from all other arcs in $\mc{A}$, and 
  $D_n\ssm\alpha_p$ has a connected component, called 
  the \emph{isolated disk}, which is a disk with 
  the puncture $p$.
\end{Definition}
}

\section{Examples}
We now give some explicit constructions for maximal good families of 
arcs, showing the lower bound in Theorem~\ref{Main_Theorem}.

We will think of $D_n$ as $\R^2$, with punctures placed on the $x$-axis at 
$\frac12,\frac32,\hdots,n-\frac12$. 
For simplicity, we will name the punctures by their 
$x$-coordinates. In this 
model, the arcs we are interested in are arcs in the plane starting 
and ending at $\infty$, where $\infty$ can be identified 
as the specified start and end puncture in the one-point 
compactification of $\R^2$.
For any $0\le a<b<c\le n$, let $\alpha_{abc}$ be the arc given by the 
graph of the function $(t-a)(t-b)(t-c)$.

Let $\mc{A}$ be the set of all arcs $\alpha_{abc}$. Obviously, 
$|\mc{A}| = \binom{n+1}{3}$.
\begin{Lemma}
  The arcs in $\mc{A}$ pairwise intersect at most twice.
\end{Lemma}
\begin{proof}
  If $(a,b,c)\neq(a',b',c')$, then the arc $\alpha_{abc}$ 
  is given by the graph of $t^3+p(t)$ and the 
  arc $\alpha_{a'b'c'}$ is given by the graph of $t^3+q(t)$ 
  where $p(t)\neq q(t)$ are two quadratic polynomials. 
  The graphs of these polynomials 
  will intersect at the solutions of $q(t)-p(t)=0$, 
  of which there are at most $2$, as $q(t)-p(t)=0$ 
  is a quadratic equation.
\end{proof}
\begin{Lemma}
  The arcs in $\mc{A}$ are essential and 
  not homotopic to each other.
\end{Lemma}
\begin{proof}
  These arcs are essential by Remark~\ref{Essential_Separates}, 
  since they separate punctures. Namely, 
  if $0\le a<b<c\le n+1$, then $\alpha_{abc}$ separates 
  $b-\frac{1}{2}$ and $b+\frac12$.

  We now show that these arcs are not homotopic to each other. 
  \begin{itemize}
    \item[Step 1.] Let $1\le i<n-1$. Let $\gamma_i$ be 
      the horizontal segment 
      between $i-\frac12$ and $i+\frac12$. Then 
      for any $\alpha\in \mc{A}$, $\alpha$ and $\gamma_i$ 
      are in minimal position, which follows from 
      Corollary~\ref{One_Intersection_Lemma}.
    \item[Step 2.] We now show that if 
      $|\alpha_{abc}\cap \gamma_i|=|\alpha_{a'b'c'}\cap \gamma_i|$ for 
      all $i$ then one of the following holds:
      \begin{enumerate}
        \item $(a,b,c)=(a',b',c')$, or
        \item $a=0$, $b=a'$, $c=b'$ and $c'=n$
      \end{enumerate}

      Note that for $1\le i\le n-1$, $|\alpha_{abc}\cap \gamma_i|=1$ 
      if and only if $i\in \{a,b,c\}$. Thus, if $0<a<b<c<n$ then 
      $\alpha_{abc}$ is uniquely determined by these 
      intersection numbers. Otherwise, it might happen that 
      only two of these intersection numbers are nonzero, 
      in which case, when condition $(1)$ does not hold, 
      then $(2)$ must. Finally, if only one of the 
      intersection numbers $|\alpha_{abc}\cap \gamma_i|$ 
      is nonzero, by the condition that 
      $a<b<c$, it follows that $a=0$, $c=n$, and 
      $b=i$, uniquely determining $\alpha_{abc}$. 
    \item[Step 3.] By Step 2, to show that 
      $\alpha_{abc}\not\sim \alpha_{a'b'c'}$, it suffices 
      to check the case where $a=0$, $b=a'$, $c=b'$ and $c'=n$. In this 
      case, there are two half-bigons and a bigon formed 
      between $\alpha_{abc}$ and $\alpha_{a'b'c'}$, 
      and each contains a puncture (see 
      Figure~\ref{fig:Example_Bigons}), and hence by 
      Lemma~\ref{Bigon_Criterion}, they are 
      not homotopic to each other.
  \end{itemize}
\end{proof}
\begin{center}
  \begin{figure}
    \includegraphics[width=0.25\linewidth]{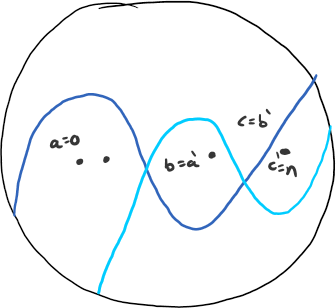}
    \caption{Showing that if $a=0,b=a',c=b'$ and $c'=n$ then 
    $\alpha_{abc}\not\sim \alpha_{a'b'c'}$.}
    \label{fig:Example_Bigons}
  \end{figure}
\end{center}

\section{Square Complexes and Theorem~\ref{Square_Cplx_Thm}}
\begin{Definition} Let $X$ be a square complex. 
  We say that two edges in a square are 
  \emph{parallel} if they are disjoint. We then define parallelism 
  between any two edges in $X$ by taking the transitive 
  closure. In other words, 
  two edges $e,e'$ in $X$ are parallel if there exists a sequence of 
  edges $e=e_1,\hdots,e_n=e'$ such that $e_i,e_{i+1}$ lie in the 
  same square and are parallel for all $0\le i<n$.
\end{Definition}
\mute{
\begin{Remark}
  Parallelism of edges is an equivalence relation.
\end{Remark}
\begin{Corollary}
  Let $\mc{A}$ be as in the proposition, and additionally, assume that 
  the arcs in $\mc{A}$ lie in general position and minimal position. Then:
  \begin{enumerate}
    \item If there is no isolated arc, then 
      $X(\mc{A})$ is a square complex homeomorphic to a disk.
    \item There is a bijective correspondence between 
      parallelism classes of edges in $X(\mc{A})$ and the arcs 
      in $\mc{A}$.
  \end{enumerate}
\end{Corollary}
\begin{proof}
  \begin{enumerate}
    \item If there are no isolated arcs, then every arc in $\mc{A}$ 
      intersects some other arc in $\mc{A}$. Let $e$ be an edge 
      in $X(\mc{A})$ connecting two regions $A,B$ separated by an arc 
      $\alpha\in\mc{A}$. Let $\beta\in\mc{A}$ be an arc intersecting 
      $\alpha$ at $x$ and forming part of the boundaries 
      of $A$ and $B$. Let $C,D$ be the other two regions 
      containing $x$. Then by construction, $A,B,C,D$ form a 
      square containing $e$.
    \item This follows from the construction of $X(\mc{A})$.
  \end{enumerate}
\end{proof}}
\smallskip
\begin{Definition} ~\label{Hyperplane_Def}
  A parallelism class of edges in $X$ is called a \emph{hyperplane}. 
  We say that an edge in $X$ is \emph{dual} to a hyperplane $h$ if it 
  is part of the parallelism class of $h$.
\end{Definition}
\begin{Remark}
  If $X$ is a planar square complex, we will also treat 
  hyperplanes in $X$ as immersed curves or arcs starting 
  and terminating in $\del X$. 
  If $h$ is a simple hyperplane, then its corresponding arc will 
  bisect every square that contains an edge in $h$.
\end{Remark}
\begin{Definition}
  The union of the cells intersecting a hyperplane $h$ 
  is called the \emph{carrier} of $h$, and will be denoted by $N(h)$. 
  Similarly, for $s\subset h$ a connected segment, we define the 
  \emph{carrier} $N(s)$ to be the union of all cells intersecting $s$.
\end{Definition}
\begin{Definition}
    Let $X$ be a square complex as in Theorem~\ref{Square_Cplx_Thm}. 
  \begin{enumerate}
    \item A square in $X$ is called a \emph{boundary square} if it 
      contains an edge $e$ which is also contained in $\del X$.
    \item We say that a hyperplane $h$ in $X$ is a 
      \emph{boundary hyperplane} if it bisects a boundary square.
    \item Assume that $X$ is not separated by a single cell. 
      A \emph{boundary (hyperplane) segment} of a hyperplane $h$ is a 
      maximal hyperplane segment of $h$ starting and ending in square 
      centers with the additional requirement that 
      every dual edge has a vertex in $\del X$. For examples, 
      see Figure~\ref{fig:Boundary_Is_Circle}.
  \end{enumerate}
\end{Definition}
\mute{\begin{Remark}
  If $X$ is not separated by squares, then every boundary hyperplane 
  corresponds to a connected component of $\del X$ under the 
  nearest-point projection to $\del X$. TODO: figure?
\end{Remark}}
\begin{Definition}
  Let $X$ be a square complex, let $S$ be a square in $X$, and let 
  $v,x\in S$ be opposite vertices. Let $u,w$ be the other vertices of $S$. 
  We can consider a new square complex, obtained from $X$ by 
  removing $S$ and identifying $vu\sim vw$ and $xu\sim xw$.
  \begin{center}
    \begin{figure}
      \includegraphics[width=0.4\linewidth]{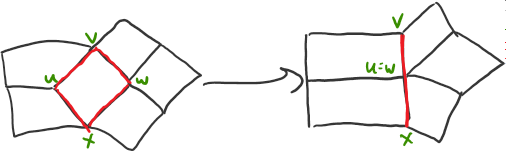}
      \caption{Folding the red square along $v$.}
      \label{fig:Folding_Example}
    \end{figure}
  \end{center}
  We say that the new complex $X'$ is obtained from $X$ by 
  \emph{folding} $S$ along $v$.
\end{Definition}
The following technical lemma will be used in the proof of Theorem~\ref{Square_Cplx_Thm}. We recommend to skip it on first reading. In the 
following, we denote by $|X|$ the number of squares in $X$. 
Recall that a square complex satisfies condition $(\ast)$ if 
its hyperplanes are simple arcs intersecting at most twice.
\begin{Lemma} \label{Reidemeister}
  Let $Y$ be a square complex with the following properties:
  \begin{enumerate}
    \item $Y$ is homeomorphic to a disk,
    \item $Y$ satisfies condition $(\ast)$,
    \item $Y$ has a hyperplane $h$ such that $N(h)$ consists of two 
      boundary squares as in Figure~\ref{fig:Boundary_Hyperplane_Eg},
      \begin{center}
        \begin{figure}
          \includegraphics[width=0.1\linewidth]{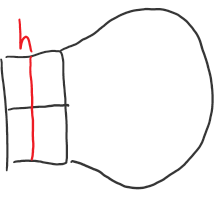}
          \caption{Describing $h$ in condition (3).}
          \label{fig:Boundary_Hyperplane_Eg}
        \end{figure}
      \end{center}
    \item $Y$ has no corner outside of $N(h)$,
    \item The hyperplanes intersecting $h$ form a bigon disjoint 
      from the boundary squares of $Y$.
  \end{enumerate}
  Then there exists a square complex $Y'$ with $|Y'|<|Y|$ such that:
  \begin{enumerate}
    \item $Y'$ is homeomorphic to a disk,
    \item $Y'$ satisfies condition $(\ast)$,
    \item $Y'$ has a hyperplane, $h'$ such that $N(h')$ consists 
      of two boundary squares as in 
      Figure~\ref{fig:Boundary_Hyperplane_Eg},
    \item $Y'$ has no corner outside of $N(h')$,
    \item The two hyperplanes intersecting $h'$ are disjoint.
  \end{enumerate}
\end{Lemma}
\begin{proof} 
  Let $\alpha,\beta$ be two hyperplanes in $Y$ which form a bigon $B$ 
  disjoint from the boundary squares of $Y$. We define 
  $\mc{B}$ as the planar square complex 
  composed of all cells intersecting $B$ in $Y$. The 
  square complex $\mc{B}$ will be called \emph{minimal} if there 
  are no hyperplane bigons contained in it distinct from $B$. 
  Suppose $\mc{B}$ is minimal. Let $S_1,S_2$ be the squares of $\mc{B}$ 
  in which $\alpha$ and $\beta$ intersect. Let $v_1,v_2$ be vertices 
  of $S_1,S_2$ such that neither the vertices $v_i$, nor their 
  opposite vertices in $S_i$ lie in $B$. We define a new 
  square complex, $Y^{*}$ by folding $S_1,S_2$ along $v_1,v_2$ 
  respectively, as in Figure~\ref{fig:Folding_Y}
  \begin{center}
    \begin{figure}
      \includegraphics[width=0.4\linewidth]{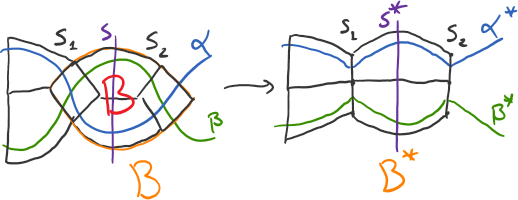}
      \caption{Folding along $S_1,S_2$ to get $Y^*$ from $Y$.}
      \label{fig:Folding_Y}
    \end{figure}
  \end{center}
  \begin{Claim} \label{Induction_Bigon}
    $Y^{*}$ is a square complex which satisfies conditions (1)-(4) 
    in the statement of Lemma~\ref{Reidemeister}.
  \end{Claim}
  \begin{proof}
    Conditions (1),(3),(4) are immediate, since the folding occurred on the 
    interior of $Y$. It remains to show condition (2). Note 
    that after folding, all hyperplane intersection numbers 
    remain the same apart from the ones for 
    $\alpha$ and $\beta$. After folding, 
    the remaining hyperplane segments of $\alpha$ and $\beta$ 
    glue together to form two new hyperplanes, $\alpha^*$ and $\beta^*$ as 
    in Figure~\ref{fig:Folding_Y}. $\alpha^{*}$ is simple, because any 
    self-intersection would have to come from a third intersection 
    between $\alpha$ and $\beta$. 
    Additionally, note that $\alpha^{*}$ and $\beta^{*}$ are disjoint. 
    We now need to show that $|\alpha^{*}\cap s^{*}|\le 2$ for 
    any hyperplane $s^{*}$ in $Y^{*}$ originating from a hyperplane 
    $s$ in $Y$ (see Figure~\ref{fig:Folding_Y}). Let 
    $\mc{B}^*$ be the square complex $\mc{B}$ after folding $S_1,S_2$. 
    If $s$ does not enter $\mc{B}$ in $Y$, then, because 
    $\alpha=\alpha^{*}$ on $Y\ssm \mc{B} = Y^{*}\ssm \mc{B}^*$, 
    it follows that $|\alpha^{*}\cap s^{*}|\le 2$. 
    Since $\mc{B}$ is minimal, it follows that if $s$ enters $\mc{B}$ by 
    intersecting $\alpha$, it must exit $\mc{B}$ by intersecing 
    $\beta$. After passing to $\alpha^{*}$ and $\beta^*$, we still get that 
    $\alpha^{*},\beta^{*}$ intsersect $s^{*}$ once in 
    $\mc{B}^{*}$, and hence, by condition (2) of $Y$, we get that 
    $Y^{*}$ satisfies condition (2). This proves the claim.
  \end{proof} 
  We return to the proof of Lemma~\ref{Reidemeister}. 
  Let $\alpha_0,\beta_0$ be the two hyperplanes intersecting $h$, 
  and let $B_0$ be the bigon they form, and $\mc{B}_0$ be the 
  respective square complex. 
  Since $Y$ is finite, there must exist a minimal bigon complex $\mc{B}$ 
  contained in $\mc{B}_0$. By the claim, we get $Y^{*}$, 
  which is a square complex with one fewer bigons contined in 
  $\mc{B}_0$. We can then continue to remove the 
  bigons in $\mc{B}_0$ until $\mc{B}_0$ is 
  minimal. When $\mc{B}_0$ is minimal, we can remove it, giving our 
  desired complex $Y'$, which will also satisfy condition (5) in addition 
  to (1)-(4).
\end{proof}
\begin{Definition}
  Let $X$ be a square complex homeomorphic to a disk, and let 
  $h$ be a simple hyperplane in $X$. Consider the square complex 
  $X'$ obtained from $X$ by removing $N(h)$, 
  and gluing along opposite edges in $\del N(h)$. 
  We say that $X'$ is obtained from $X$ by \emph{collapsing} along $h$. 
\end{Definition}
\begin{proof} [Proof of Theorem~\ref{Square_Cplx_Thm}]
  If $X$ contains a spur, then we are done. Otherwise, 
  we collapse all hyperplanes containing $1$-cells not 
  contained in $2$-cells to obtain a new complex $X'$. Obviously, 
  a corner of $X'$ corresponds to a corner of $X$, 
  and the hyperplanes of $X'$ satisfy condition $(\ast)$.

  It remains to prove the theorem under the hypothesis 
  that every $1$-cell in $X$ is contained in a $2$-cell. We 
  do this by induction on the number of squares in $X$.
  \begin{Claim}
    If $X$ contains a separating vertex, then it has a corner
  \end{Claim} 
  \begin{proof}
    Let $X_1,\hdots, X_n$ be the connected components of 
    $X\ssm \{v\}$. By induction, there is a corner in every 
    $X_i$. This corner is either a corner of $X$, and we are done, 
    or a corner at $v$. Let $X_1$ be the component with the 
    smallest number of squares, and let $S_v$ be the square 
    in $X_1$ containing $v$. We define $\hat X$ as the 
    double of $X_1\ssm S_v$ along the two remaining edges 
    $e_1,e_2$ from $S_v$. Let $\hat e_1,\hat e_2$ be the 
    edges corresponding to $e_1,e_2$ in $\hat X$.
    Note that the number of squares 
    in $\hat X$, satisfies $|\hat X|=2|X_1|-2<|X|$. 
    \begin{center}
      \begin{figure}
        \includegraphics[width=0.5\linewidth]{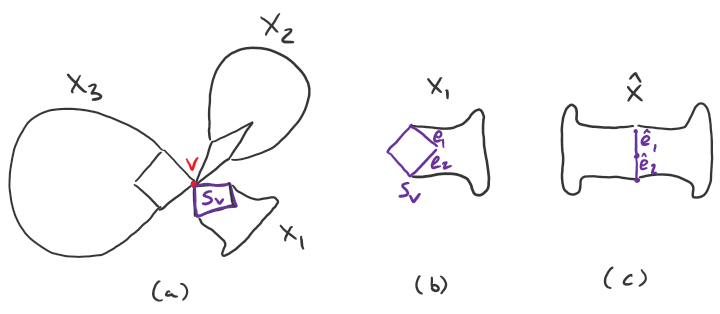}
        \caption{An example of the above construction. $(a)$ 
        illustrates the complex $X$, $(b)$ illustrates the smallest 
        component $X_1$ of $X\ssm v$, and 
        $(c)$ illustrates the constructed complex $\hat X$.}
        \label{fig:SepVertex}
      \end{figure}
    \end{center}
    \begin{Remark}
      It cannot happen that a hyperplane in $\hat X$ passes through 
      $\hat e_1$ and $\hat e_2$. This would correspond 
      to a non-simple hyperplane in $X$. In particular, it follows 
      that all hyperplanes in $\hat X$ are simple arcs.
    \end{Remark}
    We now justify 
    why the hyperplanes in $\hat X$ satisfy condition $(\ast)$. 
    Let $\hat\alpha,\hat\beta$ be hyperplanes in the 
    double $\hat X=(X\ssm S_v)\sqcup (X\ssm S_v)/\sim$.

    There are four cases we must check:
    \begin{itemize}
      \item $\hat\alpha,\hat\beta$ lie in the first copy of $(X\ssm S_v)$. 
        In this case, $\hat X$ satisfies condition $(\ast)$ because 
        $X$ satisfies it.
      \item $\hat\alpha,\hat\beta$ lie in different copies of 
        $(X\ssm S_v)$. In this case, $\hat \alpha\cap\hat\beta=\emptyset$, 
        so condition $(\ast)$ is automatically satisfied.
      \item $\hat\alpha$ lies in both copies, $\hat\beta$ lies in the 
        first. In this case, by the remark, it follows without 
        loss of generality that $\hat\alpha$ passes through $\hat e_1$ 
        and not $\hat e_2$. Thus, in the copy of $X\ssm S_v$ 
        containing $\hat \beta$, 
        $\hat \alpha$ coincides with its corresponding hyperplane, 
        $\alpha$ in $X\ssm S_v$. Since $\hat\beta$ coincides with its 
        corresponding hyperplane $\beta$ in $X\ssm S_v$, 
        by condition $(\ast)$ for $X$, we get that $\hat\alpha,\hat\beta$ 
        satisfy condition $(\ast)$.
      \item $\hat\alpha,\hat\beta$ both lie in both 
        copies of $(X\ssm S_v)$. This case can only happen if 
        $\hat\alpha,\hat\beta$ are obtained as doubles of 
        hyperplanes $\alpha,\beta$ passing through $e_1,e_2$ respectively. 
        If $\hat\alpha\neq\hat\beta$, then $\alpha,\beta$ must cross in 
        $S_v$, meaning that in $(X\ssm S_v)$, we have 
        $|\hat\alpha\cap \hat\beta|\le 2-1=1$ (see 
        Figure~\ref{fig:SepVertex2}). Thus, doubling 
        these hyperplanes gives $|\hat\alpha\cap\hat\beta|\le 1+1=2$. 
    \end{itemize}
    Thus, by induction, $\hat X$ contains a corner, 
    which gives rise to a corner in $X$. This justifies 
    that there is no separating vertex.
    \begin{center}
      \begin{figure}
        \includegraphics[width=0.3\linewidth]{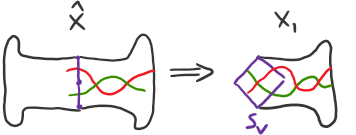}
        \caption{An example of the contradiction in the fourth bullet 
        above.}
        \label{fig:SepVertex2}
      \end{figure}
    \end{center}
  \end{proof}
  If $X$ contains a separating square $S$, then let 
  $X_1$ be the smallest component of $X\ssm S$. Let 
  $\hat X$ be the double of $X_1$ along $X_1\cap S_v$, 
  which might consist of one, two, or three edges. The 
  argument justifying condition ($\ast$) is as above, and the 
  doubling works in the same way. Hence we can assume that 
  $X$ does not contain separating squares. Note that 
  likewise, there are no separating edges, as they would 
  imply the existence of a corner or a separating square.

  We will attempt to reproduce this doubling 
  trick in the remaining case, where $X$ does not contain 
  a separating cell. We now assume by contradiction that 
  $X$ has no corners, and we will show that there exists a smaller 
  complex $Z$ satisfying the hypothesis with no corners.

  Since $X$ contains no separating cells, $X$ is homeomorphic 
  to a disk, and the union of the boundary segments is a circle 
  (see Figure~\ref{fig:Boundary_Is_Circle}). 
  \begin{center}
    \begin{figure}
      \includegraphics[width=0.2\linewidth]{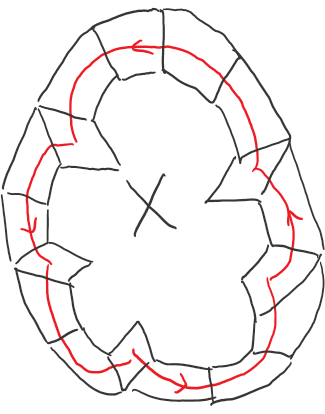}
      \caption{Since $X$ is homeomorphic to a disk with no 
      separating cells, the 
      boundary segments form a circle.}
      \label{fig:Boundary_Is_Circle}
    \end{figure}
  \end{center}
  We choose an orientation on this circle which induces 
  an orientation on each boundary segment.

  This means that every boundary segment $s\subset h$ induces 
  a decomposition of $h$ as an oriented arc $h= s_{-} s s_{+}$. Note 
  that the orientation induced on $h$ in the above way depends 
  on the boundary segment $s$ (see Figure~\ref{fig:Hyperplane_Orientation}).

  \begin{center}
    \begin{figure}
      \includegraphics[width=0.2\linewidth]{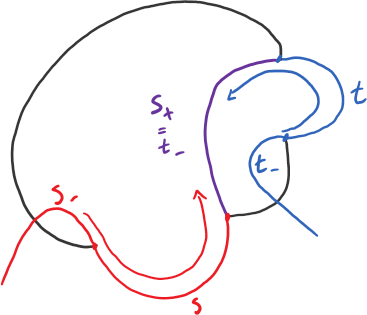}
      \caption{The orientation on each hyperplane depends on the 
      boundary segment.}
      \label{fig:Hyperplane_Orientation}
    \end{figure}
  \end{center}

  Let $\mc{S}$ denote the set of all boundary segments. For every 
  $s\in \mc{S}$, we denote $s_{-},s_{+}$ as above. 
  $X\ssm N(s_{-})$ has two (possibly disconnected) sides. 
  Let $\hat X$ be the side not containing 
  $s$. Let $X_{s_{-}}=\hat X\cup N(s_{-})$ as 
  in Figure~\ref{fig:S_Minus_Complex}. We analogously define 
  $X_{s_{+}}$. 
  \begin{center}
    \begin{figure}
      \includegraphics[width=0.15\linewidth]{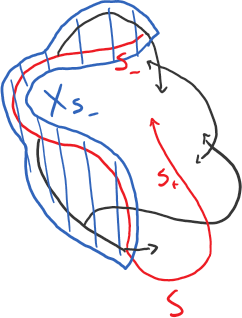}
      \caption{The complex $X_{s_{-}}$.}
      \label{fig:S_Minus_Complex}
    \end{figure}
  \end{center}
  Let $M=\max_{s\in\mc{S}} \max\{|X_{s_{-}}|,|X_{s_{+}}|\}$. Without 
  loss of generality, assume that $M = |X_{s_{-}}|$ for 
  some $s\in \mc{S}$. Set $X'=X_{s_{+}}$.
  \begin{Remark} 
    By choice of $s$, $|X'|\le \frac{|X|}{2}$.
  \end{Remark}

  \begin{Claim}
    If $t$ is a boundary segment in $X'$, then it cannot 
    happen that both $t_{-}$ and $t_{+}$ intersect $s_{+}$.
  \end{Claim}
  \begin{proof}
    If $t_{+}$ and $t_{-}$ both intersect $s_{+}$, then by condition 
    $(\ast)$, it follows that neither of them may intersect $N(s_{-})$. 
    In particular, $t_{-}\cap N(s_{-})=\emptyset$, so 
    $X_{s_{-}}\subset X_{t_{-}}$ as in 
    Figure~\ref{fig:Max_S_Complex}. This containment is 
    strict, since $N(t_{-})\subset X_{t_{-}}\ssm N(s_{-})$. 
    This contradicts maximality of $|X_{s_{-}}|$.
    \begin{center}
      \begin{figure}
        \includegraphics[width=0.2\linewidth]{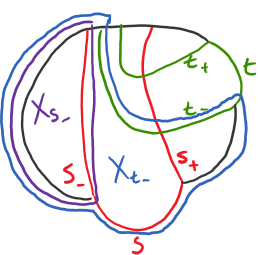}
        \caption{If $s$ is chosen such that $X_{s_{-}}$ is maximal, 
        then it cannot happen that both $t_{-}$ and $t_{+}$ intersect 
        $s_{+}$.}
        \label{fig:Max_S_Complex}
      \end{figure}
    \end{center}
  \end{proof}
  By the claim, we can label each boundary 
  segment $t\neq s_{+}$ in $X'$ by $+$, $-$, or $0$ depending on 
  whether $t_{+}$, $t_{-}$, or neither intersect $s_{+}$. 
  We order the boundary segments of $X'$ different from $s_{+}$ using 
  the orientation induced from $\del X$.
  Note that the first and last boundary segments in this order are 
  labelled by $-$ and $+$ respectively. Let $t^1$ be the last 
  boundary segment labelled by $-$. Then we take 
  the consecutive boundary segments $t^2,t^3,\hdots,t^{k-1}$ 
  as long as they are labelled by $0$. This gives a 
  a sequence of consecutive segments, 
  $t^1,\hdots,t^k$ labelled $-,0,\cdots,0,+$ (see 
  Figure~\ref{fig:Segment_Labelling}).

  \begin{center}
    \begin{figure}
      \includegraphics[width=0.3\linewidth]{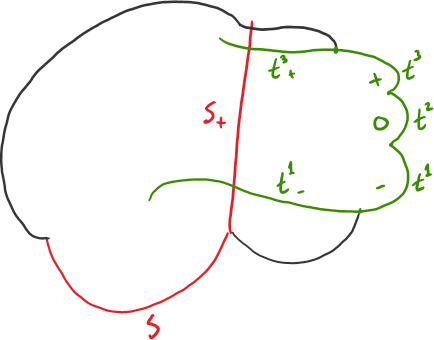}
      \caption{Finding a sequence of consecutive segments labelled 
      $-,0,\cdots,0,+$.}
      \label{fig:Segment_Labelling}
    \end{figure}
  \end{center}
  We now examine $t^1_{-}$ and $t^k_{+}$ and treat 
  three cases, each time constructing $Z$ differently. 

  \begin{enumerate}[label=(\alph*)]
    \item If $t^1_{-}$ and $t^k_{+}$ intersect 
      each other before intersecting $s_{+}$, then we can look at 
      the subcomplex $Y$ of $X'$ bounded by $N(t^1_{-})\cup N(t^k_{+})
      \bigcup_{1\le i\le k} N(t^i)$ as in Figure~\ref{fig:Int_Before_s}. 
      \begin{center}
        \begin{figure}
          \includegraphics[width=0.3\linewidth]{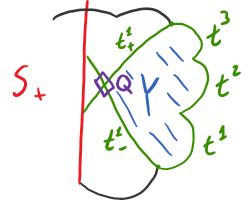}
          \caption{Defining the complex $Y$, as the union of 
          all the squares intersecting the 
          region $Y$ in the figure.}
          \label{fig:Int_Before_s}
        \end{figure}
      \end{center}
      Consider the square 
      $Q=N(t^1_{-})\cap N(t^k_{+})$, and let $\hat Y=Y\ssm Q$. 
      let $Z$ be the double of $\hat Y$ along the remaining 
      edges of $Q$. We also know that $|Z|\le 2|Y|-2\le 2|X'|-2<|X|$.
    \item If $t^1_{-}$ and $t^k_{+}$ 
      do not intersect each other before intersecting $s_{+}$, 
      and $|t^1_{+}\cap t^k_{-}|\le1$, as in 
      Figure~\ref{fig:No_Bigon_Germs}, 
      then we define $Y$ as the complex bounded by 
      $N(s_{+})\cup N(t^1_{-})\cup N(t^k_{+})\bigcup_{1\le i\le k} N(t^i)$. 
      \begin{center}
        \begin{figure}
          \includegraphics[width=0.2\linewidth]{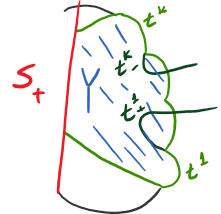}
          \caption{The second case, where the segments $t^1_{+}$ and 
          $t^k_{-}$ are disjoint or intersect once.}
          \label{fig:No_Bigon_Germs}
        \end{figure}
      \end{center}
      We keep the notation $s_{+}$ for the 
      segment of $s_{+}$ between the intersections $s_{+}\cap t^1_{-}$ 
      and $s_{+}\cap t^k_{+}$. 

      Let $\hat Y$ be the complex obtained from $Y$ by 
      collapsing all non-boundary hyperplanes. This operation does 
      not produce corners and keeps $\hat Y$ homeomorphic to a disk 
      (see Figure~\ref{fig:Removing_Nonboundary}).
      \begin{center}
        \begin{figure}
          \includegraphics[width=0.5\linewidth]{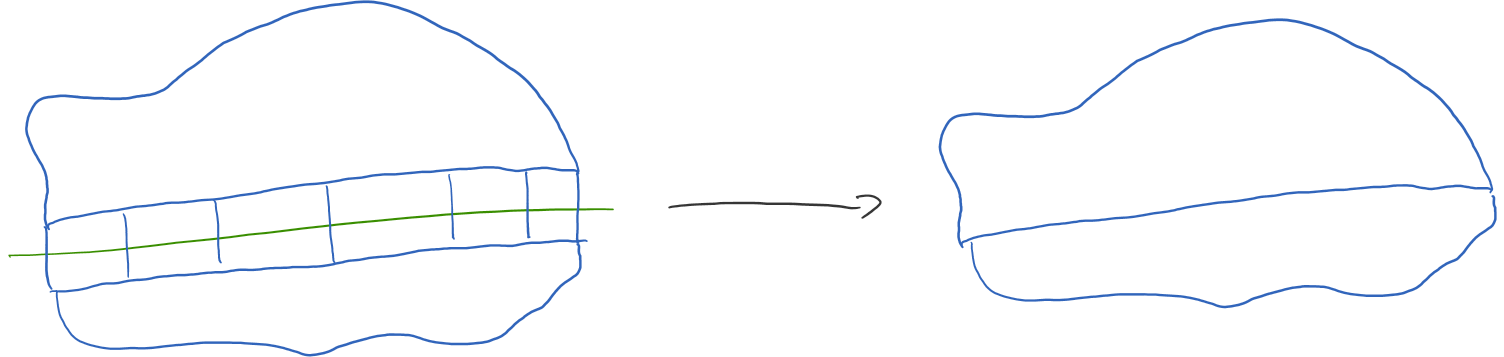}
          \caption{The construction of $\hat Y$ by removing nonboundary 
          hyperplanes does not change the homeomorphism type of 
          the complex.}
          \label{fig:Removing_Nonboundary}
        \end{figure}
      \end{center}
      The boundary hyperplanes of $Y$ are the same as those for 
      $\hat Y$, so if $s_{+},h^1,\hdots,h^k$ 
      were the boundary hyperplanes of $Y$, we denote the boundary 
      hyperplanes of $\hat Y$ by $\hat s_{+},\hat h^1,\hdots, \hat h^k$ 

      Doubling $\hat Y$ along $\hat s_{+}$, and then collapsing 
      $N(\hat s_{+})$ gives us $Z$. As in the previous case, 
      $|Z|=2|\hat Y|-4\le 2|Y|-4<|X|$.
    \item If $t^1_{-}$ and $t^k_{+}$ 
      do not intersect each other before intersecting $s_{+}$, 
      and $t^1_{+}$, $t^k_{-}$ form a bigon, as in 
      Figure~\ref{fig:Bigon_Germs}, then we define $Y$, 
      $\hat Y$ exactly as above.
      \begin{center}
        \begin{figure}
          \includegraphics[width=0.2\linewidth]{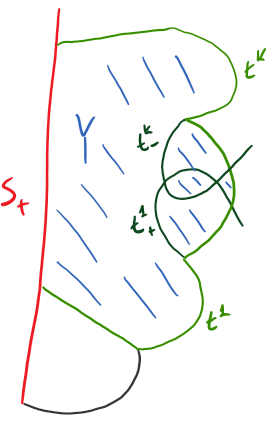}
          \caption{The third case, where the segments $t^1_{+}$ and 
          $t^k_{-}$ form a bigon.}
          \label{fig:Bigon_Germs}
        \end{figure}
      \end{center}

      We verify that $\hat Y$ satisfies the conditions of 
      Lemma~\ref{Reidemeister}.
      \begin{enumerate}[label=(\arabic*)]
        \item Follows from Figure~\ref{fig:Removing_Nonboundary},
        \item Immediate,
        \item $\hat s_{+}$ is this hyperplane, 
          with only $\hat h_1,\hat h_k$ intersecting it,
        \item Since $X$ had no corner, $Y$ has exactly two corners 
          which lie in $N(s_{+})$. Since collapsing nonboundary 
          hyperplanes does not introduce new corners, the same holds 
          for $\hat Y$.
        \item The bigon formed by $h^1,h^k$ must not intersect the 
          boundary squares of $X'$, 
          or else we would get $t^j$ labelled $+$ or 
          $-$ for $1<j<k$. This property then extends to $Y$ and 
          $\hat Y$.
      \end{enumerate}
      By Lemma~\ref{Reidemeister}, we can replace $\hat Y$ with a 
      complex $Y'$ with no corners outside of $N(\hat s_{+})$, 
      and reduce to the second case, defining $Z$ as above.
  \end{enumerate}
  In all three cases, we have $|Z|<|X|$, so by the induction hypothesis, 
  $Z$ has a corner at a vertex $v$. Note that $v$ cannot belong 
  to the edges along which we doubled. Consequently, it must belong 
  to $Y\ssm Q$ in the first case or $\hat Y\ssm N(\hat s_{+})$ in the 
  second case. This is a contradiction. 
\end{proof}

\section{Strips, bigons, and Proposition~\ref{Bigon_Count}}
In this section, we will prove Proposition~\ref{Bigon_Count}. Throughout 
this section, we assume that $\mc{A}$ is a family of arcs as in 
Proposition~\ref{Bigon_Count}. As a warm-up, we show:
\begin{Lemma} \label{Non_Isolating_Essential}
  $\alpha_p$ is the unique arc becoming 
  nonessential after removing $p$.
\end{Lemma}
\begin{proof}
  Let $\alpha\in \mc{A}$ be distinct from $\alpha_p$.
  Since $\alpha$ and $\alpha_p$ are disjoint, they form a strip, 
  where $p$ is outside of this strip. Note that since $\alpha\neq\alpha_p$, 
  there must be some puncture $r$ in this strip. 
  Since $\alpha$ is essential, there exists a puncture 
  $q$ on the other side of this strip, as 
  in Figure~\ref{fig:No_p_Essential}.
  \begin{center}
    \begin{figure}[ht]
      \includegraphics[width=0.25\linewidth]{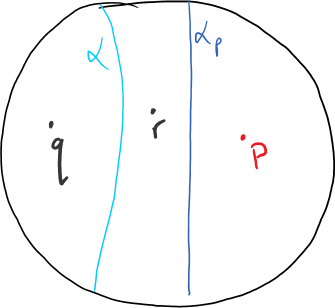}
      \caption{$\alpha$ remains essential after removing $p$.}
      \label{fig:No_p_Essential}
    \end{figure}
  \end{center}

  Now, $\alpha$ separates $q$ and $r$, and thus remains essential after 
  removing $p$.
\end{proof}
\begin{Lemma} \label{Min_Pos_Bigon}
  If two arcs $\alpha,\beta\in \mathcal{A}$ 
  become homotopic and stay essential after 
  removing $p$, then before removing $p$ they 
  must have been in one of the two configurations in Figure~\ref{fig:Punctured_Bigon}.
  \begin{center}
    \begin{figure}
      \includegraphics[width=0.4\linewidth]{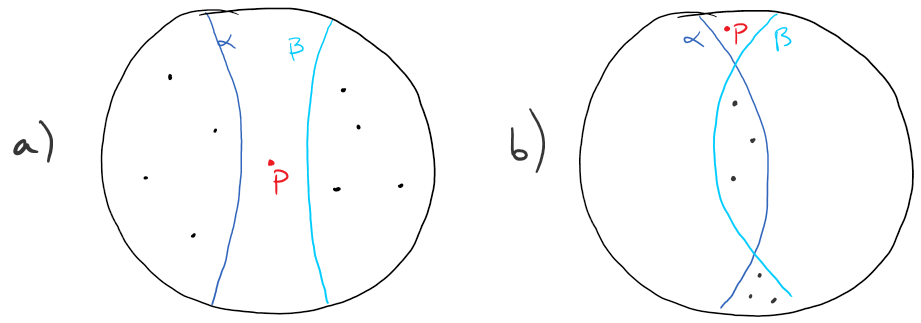}
      \caption{Two possible configurations 
      for arcs becoming homotopic after removing $p$.}
      \label{fig:Punctured_Bigon}
    \end{figure}
  \end{center}
\end{Lemma}
\begin{proof}
  We prove this by considering cases. 
  \begin{enumerate}
    \item Firstly, if $\alpha$ and $\beta$ 
      are disjoint, then they cut the disk $D_n$ into three connected 
      components, one of which is a strip bounded by $\del D_n$, $\alpha$ 
      and $\beta$. If $\alpha$ and $\beta$ were not homotopic to each 
      other before removing the puncture, then the puncture must have 
      lied in this strip. 
      Since $\alpha$ and $\beta$ are homotopic after removing 
      $p$, it follows that $p$ is the only puncture 
      in this strip. This gives Figure~\ref{fig:Punctured_Bigon}a.
    \item Secondly, we consider the case where $|\alpha\cap\beta|=1$. In 
      this case, $\alpha\cup\beta$ splits $D_n$ into four quadrants. 
      Without loss of generality, let $p$ lie in the bottom right 
      quadrant (\Romannum{4}) as in Figure~\ref{fig:Cross}:
      \begin{center}
        \begin{figure}
          \includegraphics[width=0.3\linewidth]{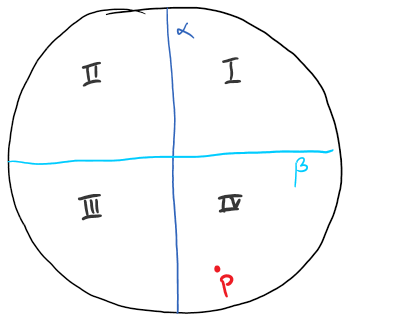}
          \caption{$\alpha$ and $\beta$ split $D_n$ into four 
          quadrants.}
          \label{fig:Cross}
        \end{figure}
      \end{center}
      By Lemma~\ref{Bigon_Criterion}, every quadrant must have a 
      puncture. However, if quadrant \Romannum{2} has a puncture $q$, 
      then quadrant \Romannum{3} cannot have a puncture $r$, since 
      $\beta$ would separate $r$ from $q$, 
      whereas $\alpha$ would not. This would then contradict 
      the fact that $\alpha$ and $\beta$ become homotopic after 
      removing $p$. 
    \item The last scenario is when $|\alpha\cap\beta|=2$. In this 
      case, $\alpha$ and $\beta$ lie as in 
      Figure~\ref{fig:Two_Intersect_ab}.
      \begin{center}
        \begin{figure}
          \includegraphics[width=.25\linewidth]{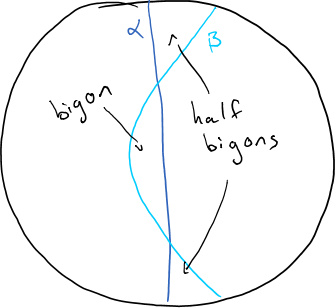}
          \caption{Simple arcs intersecting twice lie in this 
          configuration.}
          \label{fig:Two_Intersect_ab}
        \end{figure}
      \end{center}
      By Lemma~\ref{Bigon_Criterion}, there must be a puncture 
      in the bigon and the two half-bigons in 
      Figure~\ref{fig:Two_Intersect_ab}. 
      Let $r\neq p$ be a puncture in one of the half-bigons. 

      If there was a puncture $q$ outside the bigons and half-bigons
      (without loss of generality, let it lie in the right region of 
      Figure~\ref{fig:Two_Intersect_ab}), then $\beta$ 
      separates $r$ and $q$, while $\alpha$ does not, and this 
      property remains after $p$ is removed. Thus, $\alpha$ 
      and $\beta$ do not become homotopic to each other. 

      This means that the punctures in the bigon and 
      half-bigons are the only possible punctures in 
      this configuration. If $p$ lies in a bigon or half-bigon, 
      then it must be the only puncture in this bigon or 
      half-bigon, otherwise, after its removal, $\alpha$ 
      and $\beta$ would still be in minimal position, 
      and non-homotopic. If $p$ lies in the bigon, 
      then removing $p$ would make $\beta$ nullhomotopic 
      by the above discussion, giving a contradiction.

      Thus, $p$ lies in one of the half-bigons, it 
      is the only puncture in this half-bigon, and 
      all other punctures are contained in the bigon or 
      other half-bigon between $\alpha$ and $\beta$, 
      giving the configuration in Figure~\ref{fig:Punctured_Bigon}a.
  \end{enumerate}
\end{proof}
Lemma~\ref{Min_Pos_Bigon} yields the following immediate consequence:
\begin{Corollary} \label{Three_Disjoint_Arcs}
  Let $q\neq p$ be a puncture, and let $\delta\neq \delta'\in \mc{A}$ be 
  arcs becoming homotopic after removing $p$. Then 
  $|\delta\cap \delta'| = 0$ if and only if exactly one 
  of $\delta$,  $\delta'$ separate $p$ from $q$. In particular, 
  there are no triples of disjoint arcs becoming pairwise homotopic 
  after removing $p$
\end{Corollary}
\begin{Lemma}\label{lem:gamma}
  If $\mc{A}$ is maximal, then if there are two arcs 
  $\alpha,\beta\in\mc{A}$ 
  as in Figure~\ref{fig:Punctured_Bigon} $b)$, 
  then the arc $\gamma$ in Figure~\ref{fig:Gamma_Picture} belongs to 
  $\mc{A}$. 
  \begin{center}
    \begin{figure}
      \includegraphics[width=0.3\linewidth]{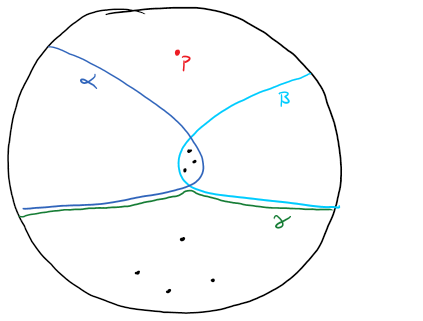}
      \caption{If $\alpha,\beta\in \mc{A}$, then the 
      arc $\gamma$ must be in $\mc{A}$.}
      \label{fig:Gamma_Picture}
    \end{figure}
  \end{center}
\end{Lemma}
We think of $\gamma$ as a subset of $\alpha\cup\beta$, 
and set $\gamma_{\alpha}=\gamma\cap\alpha$ and 
$\gamma_{\beta} = \gamma\cap\beta$. 
\begin{proof}
  We must check that if $\delta$ is an arc in $\mc{A}$, it cannot intersect 
  $\gamma$ more than twice.
  Note that by construction of $\gamma$, if $\delta$ 
  intersects $\gamma$ at any point, then it must either 
  intersect $\gamma_\alpha$ or $\gamma_\beta$.

  We assume that $|\delta\cap\gamma|\ge 3$. By the 
  pigeonhole principle, without loss of 
  generality, we can assume $|\delta\cap\gamma_{\beta}|=2$. Let 
  $\{x,y\}$ be the two intersection points, 
  and consider the segment of $\delta$ between them, 
  denoted $I_{\delta}$ as in Figure~\ref{fig:I_Delta}.
  \begin{center}
    \begin{figure}
      \includegraphics[width=0.5\linewidth]{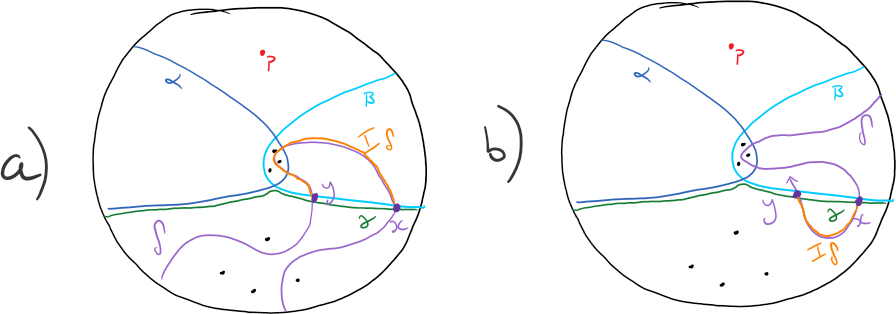}
      \caption{Drawing parts of $\delta$ in different configurations.}
      \label{fig:I_Delta}
    \end{figure}
  \end{center}
  The arc $\beta$ splits $D_n$ into two disks, one of which contains 
  $p$, and the other, denoted $D$, which does not. Suppose first that 
  $I_{\delta}\subset D$. Then since all of the punctures in 
  $D$ are enclosed by $\alpha$ 
  (see Figure~\ref{fig:Gamma_Picture}), it follows 
  that the only way in which $I_{\delta}$ will not 
  form an empty bigon with $\beta$ is if it intersects 
  $\alpha$ twice in $D$ (see Figure~\ref{fig:I_Delta} $a)$ above). 
  However, this means that $\delta$ intersects $\alpha$ twice 
  outside of $\gamma$ and $|\gamma_{\alpha}\cap\delta|=0$. Thus, 
  $|\gamma\cap\delta|=2$, contradiction.

  On the other hand, if $I_{\delta}\subset D_n\ssm D$, as in 
  Figure~\ref{fig:I_Delta} $b)$, then consider the segment 
  of $\delta$ from $x$ to $\del D_n$ not 
  containing $y$. Since $\delta$ 
  and $\beta$ are in minimal position, this segment 
  does not form an empty half-bigon with $\beta$. This means that 
  it must enter the bigon between $\alpha$ and 
  $\beta$. Thus it intersects $\alpha$ twice, 
  so as before, $|\gamma\cap \delta|=2$, contradiction.
\end{proof}
\begin{Lemma} \label{Three_Arcs_Homotopy}
  If $\mc{A}$ is maximal, then any set $S$ of arcs 
  in $\mc{A}$ which become pairwise homotopic 
  and remain essential after removing $p$ is of size at most $3$. 
  If $|S| = 3$, then the arcs in $S$ are as in 
  Figure~\ref{fig:Gamma_Picture}, and if $|S| = 2$, then they are 
  as in Figure~\ref{fig:Punctured_Bigon}a.
\end{Lemma}
\begin{proof}
  If $|S|\ge 3$, then by Corollary~\ref{Three_Disjoint_Arcs}, and 
  Lemma~\ref{Min_Pos_Bigon} it follows that there exist 
  $\alpha,\beta\in S$ that are in configuration $b)$ 
  of Figure~\ref{fig:Punctured_Bigon}, and that the arc 
  $\gamma$ described in this figure is also in $S$. We will show 
  that no fourth arc $\delta$ can be added to $S$ in this configuration.

  Using Figure~\ref{fig:Punctured_Bigon}b, 
  we split $D_n$ into five regions, 
  \Romannum{1}-\Romannum{5}:
  \begin{center}
    \begin{figure}
      \includegraphics[width=0.3\linewidth]{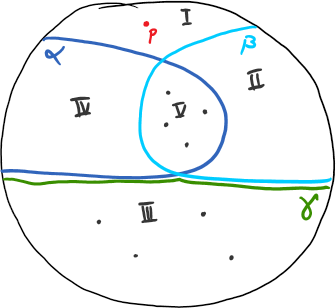}
      \caption{Regions of $D_n$, split along $\alpha$ and $\beta$.}
      \label{fig:No_p_Essential}
    \end{figure}
  \end{center}
  \begin{Claim} $|\delta\cap \gamma|=2$
  \end{Claim}
  \begin{proof}
  We justify the claim by contradiction.
    If $\delta\cap \gamma=\emptyset$, then by 
    Corollary~\ref{Three_Disjoint_Arcs}, $\delta$ lies in the 
    complement of \Romannum{3} and 
    must intersect $\alpha$ and $\beta$ twice. Note that by parity of 
    intersection numbers, $\delta$ must have both endpoints in the 
    same region, and this region cannot be \Romannum{1}, or else 
    $\delta$ would span an empty half-bigon with $\alpha$ or $\beta$. 
    Without loss of generality, let the endpoints of 
    $\delta$ lie in \Romannum{2}. In order for no half-bigons to 
    be formed, if we follow $\delta$ from these endpoints, we 
    must enter \Romannum{5}, and from there we must go 
    to \Romannum{4} (see Figure~\ref{fig:Delta_Regions}). 
    \begin{center}
    \begin{figure} \label{Delta_Regions}
      \includegraphics[width=0.3\linewidth]{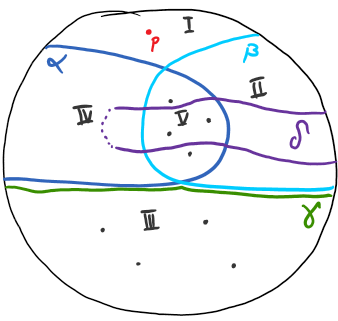}
      \caption{The arc $\delta$ forms an empty bigon with $\beta$.}
      \label{fig:Delta_Regions}
    \end{figure}
  \end{center}
    Then, having intersected $\alpha$ and $\beta$ twice, 
    they must join and form an empty bigon with $\beta$ as 
    in Figure~\ref{fig:Delta_Regions}, contradiction. 
    This justifies the claim.
  \end{proof}
  Now, since $|\delta\cap\gamma|=2$, 
  by Corollary~\ref{Three_Disjoint_Arcs}, 
  it follows that $\delta$ separates all the punctures in 
  \Romannum{5} from $p$. Since 
  $\alpha,\beta$ also separate all of the 
  punctures in \Romannum{5} from $p$, we get by 
  Corollary~\ref{Three_Disjoint_Arcs} that 
  $|\delta\cap\alpha|=|\delta\cap\beta|=0$, which is a contradiction, 
  since $\gamma\subset \alpha\cup\beta$.
\end{proof}

Let $\mc{S}$ be the set of all strips between arcs in $\mc{A}$ 
which contain the single puncture $p$. Note that by 
Lemma~\ref{Non_Isolating_Essential}, $\alpha_p$ cannot be an 
arc in such a strip.
\begin{Corollary} \label{Strip_Counting}
  Let $\mc{A}$ be a maximal good family of arcs. Assume that 
  $\alpha\in \mc{A}$ is an isolated arc, such that 
  one of the components of $D_n\ssm \alpha$ contains a single puncture 
  $p$. Let $\mc{A}'$ be the set of homotopy classes of essential arcs 
  obtained from $\mc{A}$ by removing $p$. Then
  \begin{align*}
    |\mc{A}| - |\mc{A}'| \le 1+|\mc{S}|
  \end{align*}
\end{Corollary}
\begin{proof}
  Consider the equivalence relation $\sim_p$ on $\mc{A}$ where 
  $\alpha\sim_p \beta$ if $\alpha$ and $\beta$ are homotopic 
  after removing $p$. By Lemma~\ref{Min_Pos_Bigon}, 
  Lemma~\ref{lem:gamma}, and Lemma~\ref{Three_Arcs_Homotopy}, 
  every equivalence class contains one, 
  two, or three elements, forming zero, one, or two 
  strips, respectively. Since
  $|\mc{A'}| = \{[\alpha]:\alpha\in \mc{A}\}\ssm \alpha_p$, 
  we get the result (where the $+1$ comes from $\alpha_p$).
\end{proof}
\mute{\begin{Definition}
  Let $S\in \mc{S}$ be a strip. After removing $p$, 
  a homotopy between $\alpha$ and $\beta$ gives rise to 
  an embedding of $H:[0,1]^2\to S$, where $H(0,t)=\alpha(t)$ 
  and $H(1,t)=\beta(t)$. Let $(x,y)$ be the preimage of 
  $p$ under this embedding, and consider the image of $[0,1]$ in $D_n$ 
  under the map $t\to H(x,t)$. The image passes through $p$, and 
  will be called the \emph{core} of $S$, and will be denoted by $Core(S)$.
\end{Definition}}
\begin{Lemma} \label{Strip_Intersection}
  $\alpha_p\subset \bigcap_{S\in \mc{S}} S$, and its endpoints 
  are on the same connected component of $\del D_n\cap S$ for any 
  strip $S\in\mc{S}$.
\end{Lemma}
\begin{proof}
  Let $S\in\mc{S}$ be a given strip between two arcs 
  $\alpha,\beta\in\mc{A}$. Since $\alpha_p$ is disjoint 
  from every $\alpha,\beta\in \mc{A}$, it follows 
  that it must lie in one connected component of 
  $D_n\ssm(\alpha\cup\beta)$. 

  If it does not lie in $S$, then without loss 
  of generality, it must lie in the connected component bounded 
  by $\alpha$ and $\del D_n$. In this case, since $\beta$ 
  is essential, there must be a puncture $r$ in the 
  connected component bounded by $\beta$ and $\del D_n$, and 
  $\alpha_p$ cannot separate $r$ and $p$, giving a contradiction.

  If the endpoints of $\alpha_p$ lie in different connected 
  components of $\del D_n \cap S$ for some strip $S$, 
  then $\alpha_p\sim \alpha$ or $\alpha_p\sim \beta$, since $p$ 
  is the only puncture in $S$, contradiction.
\end{proof}
\begin{Definition}
  There exists an arc $\eps$ from $p$ to $\del D_n$ 
  disjoint from $\alpha_p$.
  By Lemma~\ref{Strip_Intersection}, 
  $\eps$ lies in every $S\in \mc{S}$. If $S\in \mc{S}$ 
  then there are two unique (up to homotopy) 
  arcs in $S$ from $p$ to $\del D_n\cap S$. 
  The first arc is $\eps$, 
  and the second will be denoted by $\delta_S$
  (see Figure~\ref{fig:core_def}).
  \begin{center}
    \begin{figure}
      \includegraphics[width=0.3\linewidth]{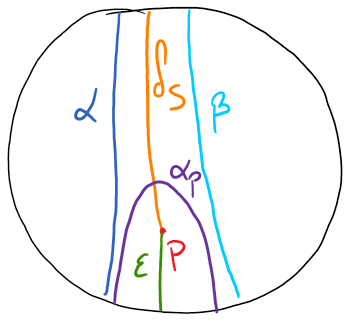}
      \caption{Defining $\delta_S$.}
      \label{fig:core_def}
    \end{figure}
  \end{center}
  We denote
  \begin{align*}
    \mc{G} = \{\delta_S\}_{S\in \mc{S}}\cup \{\eps\}
  \end{align*}
\end{Definition}
\begin{Remark}
  If $S\neq S'$ are two strips, then $\delta_S\neq \delta_{S'}$, 
  since $\delta_S$ and $\eps$ uniquely determine $S$.
\end{Remark}
\begin{Lemma} \label{Cores_Int_Once}
  $\mc{G}$ is a family of arcs 
  from $p$ to $\del D_n$ which pairwise intersect at most once.
\end{Lemma}
\begin{proof}
  Firstly, we note that by construction, $\eps$ is disjoint from 
  $\delta_S$ for any $\delta_S\in \mc{C}$. Now, let 
  $\delta_S\neq \delta_{S'}$ be two arcs in minimal 
  position coming from 
  strips $S\neq S'$. Let $\alpha,\beta\in \mc{A}$ be the sides of 
  these strips as in Figure~\ref{fig:Core_Sides}.
  \begin{center}
    \begin{figure}
      \includegraphics[width=0.3\linewidth]{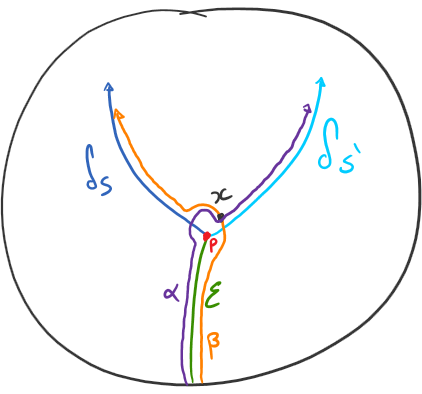}
      \caption{Taking sides of $S$ and $S'$.}
      \label{fig:Core_Sides}
    \end{figure}
  \end{center}
  It is clear that $|\alpha\cap\beta|=|\delta_S\cap\delta_{S'}|+1$, so 
  we just need to show that $\alpha$ and $\beta$ are in minimal 
  position. Let $x$ be the intersection point of $\alpha$ and 
  $\beta$ near $p$ (see Figure~\ref{fig:Core_Sides}). If there is 
  an empty bigon or half-bigon between $\alpha$ and $\beta$ 
  whose boundary does not contain $x$, then this will also 
  be a bigon or half-bigon will also be between 
  $\delta_S,\delta_{S'}$ contradicting 
  their minimal position. The half bigon containing $\eps$ (whose 
  boundary contains $x$) contains the puncture $p$, 
  and is thus nonempty. Since $\delta_S,\delta_{S'},\eps$ 
  are not pairwise homotopic and $\delta_S,\delta_{S'}$ do not 
  form an empty bigon, it follows that the other regions 
  between $\alpha$ and $\beta$ whose boundaries contain $x$ 
  are not empty bigons or half-bigons.

\end{proof}
We are now ready to prove Proposition~\ref{Bigon_Count}:
\begin{proof}
  By Corollary~\ref{Strip_Counting}, the number of arcs 
  which become pairwise homotopic after removing $p$ is 
  equal to the number of strips formed by arcs in $\mc{A}$ 
  after removing $p$. By Lemma~\ref{Cores_Int_Once}, and Theorem 
  1.7 in \cite{Przytycki}, $|\mathcal{G}|\le \binom{n}{2}$. 
  Since $|\mathcal{S}|=|\mathcal{G}|$-1, the proposition follows.
\end{proof}

\section{Discussion and Corollaries}
We present a few corollaries of Theorem~\ref{Main_Theorem} and 
Theorem~\ref{Square_Cplx_Thm}.
\begin{Corollary}
  Let $X$ be a planar square complex homotopy equivalent to 
  a disk, whose hyperplanes satisfy condition $(\ast)$. Then if $X$ 
  contains at least two $X$ has two corners and/or spurs.
\end{Corollary}
\begin{proof}
  By Theorem~\ref{Square_Cplx_Thm}, $X$ has a corner or a spur. 
  If $X$ has only one corner or spur at $v$, then since $X$ 
  contains at least two $1$-cells, it follows that 
  the double $Y=X\sqcup_v X$ will have no 
  corners or spurs. Condition $(\ast)$ 
  is immediately satisfied by $Y$, and the homotopy type of $Y$ is 
  still that of a disk, contradicting Theorem~\ref{Square_Cplx_Thm}.
\end{proof}
This corollary is reminiscent of Greendlinger's Lemma from 
small cancellation theory:
\begin{Theorem}
  Let $X$ be a $C(6)$-complex, and $D\to X$ a minimal disc diagram. 
  Then one of the following holds:
  \begin{itemize}
    \item $D$ is a single cell,
    \item $D$ is a ladder, or 
    \item $D$ has at least three spurs or shells of degree $\le 3$.
  \end{itemize}
\end{Theorem}
The main difference is that in the above corollary, we only require 
knowledge of the hyperplane intsersection data instead of 
negative curvature assumptions.
\subsection{Generalizing Theorems~\ref{Main_Theorem} 
and~\ref{Square_Cplx_Thm}}
The proof method outlined in this paper can also be used to simplify 
the case for once-intersecting families:
\begin{Theorem} [Theorem 1.7 in {\cite{Przytycki}}]
  The maximal cardinality of a good family of arcs pairwise 
  intersecting at most once on $D_n$ 
  is $\binom{n}{2}$.
\end{Theorem}
Note that in general, it is not true that when allowing for 
$k$-intersections, the maximal families are of size 
$\binom{n+k-1}{k+1}$. For example, any maximal family of 
arcs on $D_3$ pairwise intersecting at most three times 
is of size at most $4$.

As for Theorem~\ref{Square_Cplx_Thm}, it is not true when allowing 
for triple intsersections:
\begin{center}
  \begin{figure}
    \includegraphics[width=.2\linewidth]{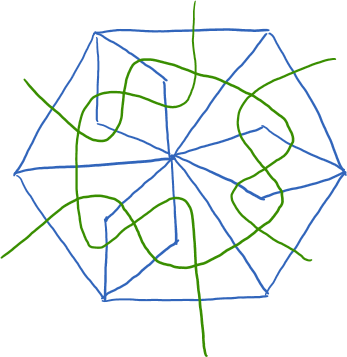}
    \caption{Theorem~\ref{Square_Cplx_Thm} is not true 
    when we allow for three intersections.}
    \label{fig:Three_Intersect}
  \end{figure}
\end{center}
\subsection{Actions of $Mod(D_n)$}
It is natural to ask whether the maximal families of 
size $\binom{n+1}{3}$ on $D_n$ are 
related by homeomorphisms in the mapping class group of $D_n$. 
Alternatively, does the mapping class group act transitively on 
maximal good families of arcs? The answer is no, as illustrated 
in Figure~\ref{fig:MCG_Action}.
\begin{center}
  \begin{figure}
    \includegraphics[width=.4\linewidth]{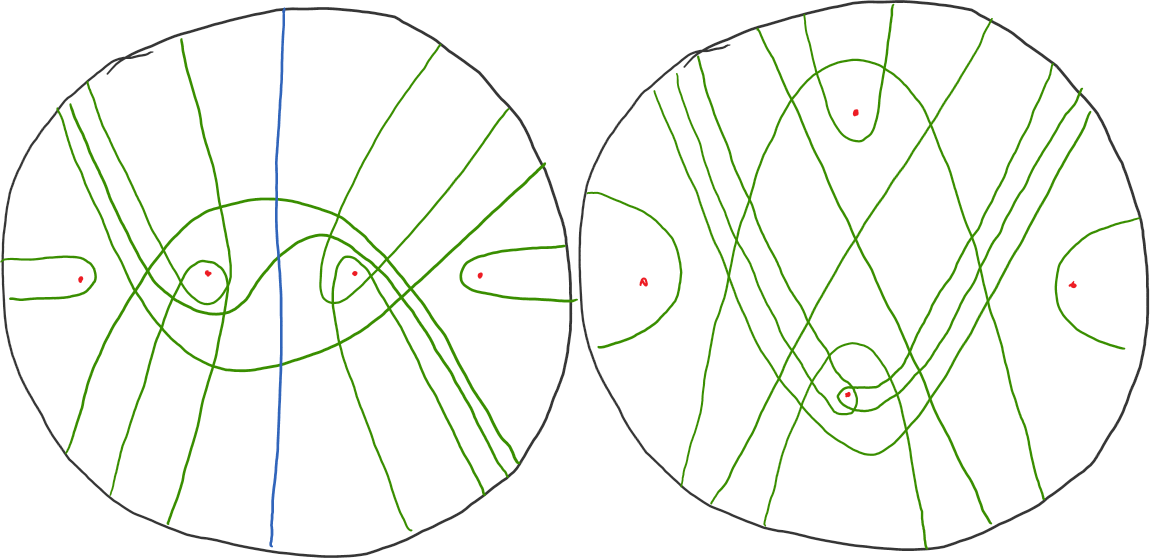}
    \caption{Two good families on $D_4$. Both are maximal, as 
    their size is $10=\binom{5}{3}$, and are in minimal position.}
    \label{fig:MCG_Action}
  \end{figure}
\end{center}
\begin{Claim}
  The good families in Figure~\ref{fig:MCG_Action} are 
  not in the same orbit of $Mod(D_n)$.
\end{Claim}
\begin{proof}
  We know that elements of the mapping class group preserve 
  intersection numbers of arcs. We first note 
  that by Lemma~\ref{Bigon_Criterion}, both of these families are 
  in minimal position. The left family in 
  Figure~\ref{fig:MCG_Action}, has an arc (the blue arc) 
  which has precisely three intersections with other arcs in 
  the family, whereas in the right family, no such arc exists. 
  Thus, these families are not obtained from one another by 
  a homeomorphism in $Mod(D_n)$.
\end{proof}
\subsection{Further Questions}
\begin{Question}
  Is every maximal good family of arcs on $D_n$ of size $\binom{n+1}{3}$? 
  This is true for the curve complex, but not much is known in 
  the case of the arc complexes.
\end{Question}
\begin{Question}
  Does the upper bound from Theorem~\ref{Main_Theorem} hold when 
  we consider arcs from $\del D_n$ to a specified puncture 
  $p\neq \del D_n$?
\end{Question}

\section*{Aknowledgements}
I would like to thank my advisor, Professor Piotr Przytycki, who has 
guided me throughout the development and writing of this paper, 
and whose support was paramount to the completion of this work.

I would also like to thank Professor Daniel Wise and the rest of the 
geometric group theory community at McGill University for their 
encouragement, support, and feedback.

Partially supported by NSERC and UMO-2015/18/M/ST1/0005

\end{document}